\documentclass[12pt,twoside,reqno]{amsart}
\usepackage[colorlinks=false,citecolor=blue]{hyperref}
\usepackage{mathptmx, amsmath, amssymb, amsfonts, amsthm, mathptmx, enumerate, color,mathrsfs}
\usepackage{graphicx}
\usepackage{lipsum}

\newcommand\blfootnote[1]{%
  \begingroup
  \renewcommand\thefootnote{}\footnote{#1}%
  \addtocounter{footnote}{-1}%
  \endgroup
}

\usepackage{epstopdf}
\newtheorem{theorem}{Theorem}[section]
\newtheorem{lemma}{Lemma}[section]
\newtheorem{remark}{Remark}[section]
\newtheorem{definition}{Definition}[section]
\newtheorem{corollary}{Corollary}[section]

\numberwithin{equation}{section}

\begin{document}
	
\title{On the Euclidean operator radius and norm}
\author{Mohammad Sababheh*, Hamid Reza Moradi and Mohammad Alomari}
\subjclass[2020]{Primary 47A12, 47A30; Secondary 47A63, 47B15.}
\keywords{Euclidean operator radius, numerical radius, norm inequality, operator matrix, off-diagonal part, operator convex function}

\begin{abstract}
In this paper, we show several bounds for the numerical radius of a Hilbert space operator in terms of the Euclidean operator norm. The obtained forms will enable us to find interesting refinements of celebrated results in the literature.

Then, the $f$-operator radius, recently defined as a generalization of the Euclidean operator radius, will be studied. Many upper bounds will be found and matched with existing results that treat the numerical radius. Special cases of this discussion will lead to some refinements and generalizations of some well-established results in the field.

Further, numerical examples are given to support our findings, and a simple optimization application will be presented.
\end{abstract}

\maketitle
\blfootnote{*Corresponding author: Mohammad Sababheh}
\pagestyle{myheadings}
\markboth{\centerline {}}
{\centerline {}}
\bigskip
\bigskip

\section{Introduction}
In the sequel, let $\mathbb{B}(\mathcal{H})$ be the $C^*$-algebra of all bounded linear operators on a Hilbert space $\mathcal{H}.$ The operator norm and the numerical radius of an operator $T\in\mathbb{B}(\mathcal{H})$ are defined, respectively, as
\[\|T\|=\sup_{\|x\|=1}\|Tx\|\;{\text{and}}\;\omega(T)=\sup_{\|x\|=1}|\left<Tx,x\right>|.\]

The investigation of possible relations between these quantities has attracted considerable attention in the literature, as one can see \cite{HSM, KMS, MS, SSM}.

The relation
\begin{equation}\label{eq_equiv_intro}
\frac{1}{2}\|T\|\leq\omega(T)\leq \|T\|
\end{equation}
represents one of the most basic relations, as found in \cite[Theorem 1.3-1]{Gustafson_Book_1997}. 

Sharpening \eqref{eq_equiv_intro} has attracted numerous researchers, where better lower and upper bounds for $\omega(T)$ in terms of $\|T\|$ is a target. The inequalities, which can be found in \cite{Ki, Ki1},
\begin{equation}\label{eq_kitt_intro1}
\omega(T)\leq\frac{1}{2}\|\;|T|+|T^*|\;\|,
\end{equation}
and
\begin{equation}\label{eq_kitt_intro2}
\frac{1}{4}\left\| |T|^2+|T^*|^2\right\|\leq \omega^2(T)\leq \frac{1}{2}\left\| |T|^2+|T^*|^2\right\|
\end{equation}
have been among the sharpest and simplest relations.

The notion of the numerical radius was extended in \cite{2} to the so-called Euclidean operator radius, which is defined for $\left( {{T}_{1}},\ldots ,{{T}_{n}} \right)\in \mathbb{B}\left( \mathscr{H} \right)\times \cdots \times \mathbb{B}\left( \mathscr{H} \right)$ as follows 
\[{{\omega }_{e}}\left( {{T}_{1}},\ldots ,{{T}_{n}} \right)=\underset{\left\| x \right\|=1}{\mathop{\sup }}\,{{\left( \sum\limits_{i=1}^{n}{{{\left| \left\langle {{T}_{i}}x,x \right\rangle  \right|}^{2}}} \right)}^{\frac{1}{2}}}.\]

The Euclidean operator radius itself was generalized in \cite{3} to the form
\[{{\omega }_{q}}\left( {{T}_{1}},\ldots ,{{T}_{n}} \right)=\underset{\left\| x \right\|=1}{\mathop{\sup }}\,{{\left( \sum\limits_{i=1}^{n}{{{\left| \left\langle {{T}_{i}}x,x \right\rangle  \right|}^{q}}} \right)}^{\frac{1}{q}}};\text{ }q\ge 1,\]
which gives $\omega_e$ upon letting $q=2.$

The definitions of $\omega_e$ and $\omega_q$ triggered the following definition, as in \cite{a1}.
\begin{definition}\cite{a1}
Let $\left( {{T}_{1}},\ldots ,{{T}_{n}} \right)\in \mathbb{B}\left( \mathscr{H} \right)\times \cdots \times \mathbb{B}\left( \mathscr{H} \right)$ and let $f:J\to \mathbb{R}$ be a strictly increasing convex  function such that $f\left( 0 \right)=0$.  
We define the $f$-operator radius of the operators $T_1,\ldots,T_n$ by
	\[{{\omega }_{f}}\left( {{T}_{1}},\ldots ,{{T}_{n}} \right)=\underset{\left\| x \right\|=1}{\mathop{\sup }}\,{{f}^{-1}}\left( \sum\limits_{i=1}^{n}{f\left( \left| \left\langle {{T}_{i}}x,x \right\rangle  \right| \right)} \right).\]
\end{definition}
We refer the reader to \cite{a1} for a detailed discussion with informative properties of $\omega_f$. It is clear that $\omega_e$ and $\omega_q$ are special cases of $\omega_f$.

Indeed, the extension of the operator norm was also discussed in the literature. For example, Popescu \cite{2} defined
\[{{\left\| \left( {{T}_{1}},\ldots ,{{T}_{n}} \right) \right\|}_{e}}:=\underset{\left( {{\lambda }_{1}},\ldots {{\lambda }_{n}} \right)\in {{\mathcal B}_{n}}}{\mathop{\sup }}\,\left\| {{\lambda }_{1}}{{T}_{1}}+\ldots +{{\lambda }_{2}}{{T}_{n}} \right\|,\]
where ${\mathcal B}_{n}$ is the set of all complex $n$-tuples $(\lambda_1,\ldots,\lambda_n)$ such that $\sum_{i=1}^{n}|\lambda_i|^2=1.$

Dragomir \cite[Theorem 6.1]{Dragomir_JIPAM_2007} showed that
\begin{equation}\label{06}
{{\left\| \left( {{T}_{1}},\ldots ,{{T}_{n}} \right) \right\|}_{e}}=\underset{\left\| x \right\|=\left\| y \right\|=1}{\mathop{\sup }}\,{{\left( \sum\limits_{i=1}^{n}{{{\left| \left\langle {{T}_{i}}x,y \right\rangle  \right|}^{2}}} \right)}^{\frac{1}{2}}}.
\end{equation}
We refer the reader to \cite{Dragomir2006, MSS} for possible readings on $\omega_e, \omega_q, \|\cdot\|_e$ and $\omega_f$.

This paper will discuss new relations between the numerical radius $\omega$ and the Euclidean operator norm $\|\cdot\|_e.$ More precisely, we will be able to present some bounds and refinements of celebrated results using $\|\cdot\|_e.$ For example, we will show that 
\[{{\omega }^{2}}\left( T \right)\le \frac{1}{2}\left\| \left( T,{{T}^{*}} \right) \right\|_{e}^{2}\le \frac{1}{2}\left\| T{{T}^{*}}+{{T}^{*}}T \right\|,\]
as such refinement of the second inequality \eqref{eq_kitt_intro2}. Many other similar results will be presented, and bounds for the numerical radius of operator matrices in terms of $\|\cdot\|_e$ will be shown. This will be the first part of our main results.

The second part of the main results will focus on further discussion of $\omega_f$, where numerous new forms that extend those relations for $\omega$ will be discussed.

The significance of this work can be summarized in the way the notions of $\omega$ and $\|\cdot\|_e$ are connected in forms that refine the existing literature about $\omega.$

The following lemmas will be needed in the subsequent sections.

\begin{lemma}\label{lem_hirz}\cite[(4.6)]{1}
Let $A,B\in\mathbb{B}(\mathcal{H})$. Then
\[\frac{1}{2}\;\underset{\theta \in \mathbb{R}}{\mathop{\sup }}\,\left\| A+{{e}^{\rm i\theta }}{{B}^{*}} \right\|=\omega \left( \left[ \begin{matrix}
   O &A  \\
   B & O  \\
\end{matrix} \right] \right),\]
where $O$ is the zero operator in $\mathbb{B}(\mathcal{H}).$
\end{lemma}

\begin{lemma}\cite[Theorem 2.3]{1}\label{lem_prod}
 Let $A,B\in\mathbb{B}(\mathcal{H})$. Then
\begin{equation*}
\max \left\{ \omega \left( {{A}^{*}}B \right),\omega \left( B{{A}^{*}} \right) \right\}
\le {{\omega }^{2}}\left( \left[ \begin{matrix}
   O & {{A}^{*}}  \\
   B & O  \\
\end{matrix} \right] \right).
\end{equation*}
\end{lemma}

\begin{lemma}\label{lem_farissi}\cite[Theorem 1.1]{Farissi_JMI_2010}
Let $f:[a,b]\to\mathbb{R}$ be a convex function, and let $0\leq \lambda \leq 1.$ Then
\[f\left(\frac{a+b}{2}\right)\leq l(\lambda)\leq \frac{1}{b-a}\int_{a}^{b}f(t)dt\leq\frac{f(a)+f(b)}{2},\]
where 
\[l(\lambda)=(1-\lambda) f\left(\frac{(1-\lambda)a+(1+\lambda)b}{2}\right)+\lambda f\left(\frac{(2-\lambda)a+\lambda b}{2}\right).\]

\end{lemma}
Following the proof of Lemma \ref{lem_farissi}, as in \cite{Farissi_JMI_2010}, we can show the operator version of the lemma.

\begin{lemma}\cite{2}\label{lem_norme}
Let $T_1,\ldots,T_n\in\mathbb{B}(\mathcal{H})$. Then
\begin{equation}\label{001}
{{\left\| \left( {{T}_{1}},\ldots ,{{T}_{n}} \right) \right\|}_{e}}\le {{\left\| \sum\limits_{i=1}^{n}{{{T}_{i}}T_{i}^{*}} \right\|}^{\frac{1}{2}}}.
\end{equation}
\end{lemma}

\begin{lemma}\cite{Kato_MathAnn_1952}\label{lemma_mixed}
(Mixed Schwarz inequality) Let $T\in\mathbb{B}(\mathcal{H})$, and let $x,y\in\mathcal{H}.$ If $0\leq \alpha\leq 1$, then
\begin{equation*} 
{{\left| \left\langle Tx,y \right\rangle  \right|}^{2}}\le \left\langle {{\left| T \right|}^{2\left( 1-\alpha \right)}}x,x \right\rangle \left\langle {{\left| {{T}^{*}} \right|}^{2\alpha}}y,y \right\rangle.
\end{equation*}
\end{lemma}

\section{The numerical radius and the Euclidean operator norm}

This section presents some interesting relations between $\omega(\cdot)$ and $\|\cdot\|_e.$ The significance of these results is reflected by how they refine some celebrated results from the literature, as we shall see.

In the first result, we present an upper bound for the numerical radius of an off-diagonal operator matrix in terms of $\|\cdot\|_e.$ Then we will list some consequences of this bound. After that, we present a reversed version of this result.

\begin{theorem}\label{0}
Let ${{T}_{1}},{{T}_{2}}\in \mathbb B\left( \mathcal H \right)$. Then
\[\omega \left( \left[ \begin{matrix}
   O & {{T}_{1}}  \\
   T_{2}^{*} & O  \\
\end{matrix} \right] \right)\le \frac{1}{\sqrt{2}}\;{{\left\| \left( {{T}_{1}},{{T}_{2}} \right) \right\|}_{e}}.\]
\end{theorem}
\begin{proof}
Since $\frac{\sqrt{2}}{2}\left| a+b \right|\le \sqrt{{{a}^{2}}+{{b}^{2}}}$ for any $a,b\in \mathbb{R}$, we have
\begin{equation}\label{1}
\begin{aligned}
{{\left( {{\left| \left\langle {{T}_{1}}x,y \right\rangle  \right|}^{2}}+{{\left| \left\langle {{T}_{2}}x,y \right\rangle  \right|}^{2}} \right)}^{\frac{1}{2}}}&\ge \frac{\sqrt{2}}{2}\left( \left| \left\langle {{T}_{1}}x,y \right\rangle  \right|+\left| \left\langle {{T}_{2}}x,y \right\rangle  \right| \right) \\ 
 & \ge \frac{\sqrt{2}}{2}\left| \left\langle \left( {{T}_{1}}+{{T}_{2}} \right)x,y \right\rangle  \right|  
\end{aligned}
\end{equation}
where the second inequality is obtained from the triangle inequality. Consequently, we have
\[\frac{\sqrt{2}}{2}\left| \left\langle \left( {{T}_{1}}+{{T}_{2}} \right)x,y \right\rangle  \right|\le {{\left( {{\left| \left\langle {{T}_{1}}x,y \right\rangle  \right|}^{2}}+{{\left| \left\langle {{T}_{2}}x,y \right\rangle  \right|}^{2}} \right)}^{\frac{1}{2}}}.\]
Substituting ${{T}_{2}}$ by ${{e}^{{\rm i}\theta }}{{T}_{2}}$, in the above inequality, we conclude that
\[\frac{\sqrt{2}}{2}\left| \left\langle \left( {{T}_{1}}+{{e}^{{\rm i}\theta }}{{T}_{2}} \right)x,y \right\rangle  \right|\le {{\left( {{\left| \left\langle {{T}_{1}}x,y \right\rangle  \right|}^{2}}+{{\left| \left\langle {{T}_{2}}x,y \right\rangle  \right|}^{2}} \right)}^{\frac{1}{2}}}\]
Now, by taking supremum over $x,y \in \mathcal H$ with $\left\| x \right\|=\left\| y \right\|=1$, we obtain
\[\frac{\sqrt{2}}{2}\left\| {{T}_{1}}+{{e}^{{\rm i}\theta }}{{T}_{2}} \right\|\le {{\left\| \left( {{T}_{1}},{{T}_{2}} \right) \right\|}_{e}}\]
due to \eqref{06}. Taking supremum over $\theta \in \mathbb{R}$, then implementing Lemma \ref{lem_hirz}, implies
\[\sqrt{2}\;\omega \left( \left[ \begin{matrix}
   O & {{T}_{1}}  \\
   T_{2}^{*} & O  \\
\end{matrix} \right] \right)\le {{\left\| \left( {{T}_{1}},{{T}_{2}} \right) \right\|}_{e}},\]
as required.
\end{proof}

An astonishing consequence of Theorem \ref{0} is the refinement of the second inequality in \eqref{eq_kitt_intro2}.
\begin{corollary}\label{02}
Let $T\in \mathbb B\left( \mathcal H \right)$. Then
\[{{\omega }^{2}}\left( T \right)\le \frac{1}{2}\left\| \left( T,{{T}^{*}} \right) \right\|_{e}^{2}\le \frac{1}{2}\left\| T{{T}^{*}}+{{T}^{*}}T \right\|.\]
\end{corollary}
\begin{proof}
By Theorem \ref{0} and Lemma \ref{lem_norme}, we infer that
\begin{equation*}
{{\omega }^{2}}\left( \left[ \begin{matrix}
   O & {{T}_{1}}  \\
   T_{2}^{*} & O  \\
\end{matrix} \right] \right)\le \frac{1}{2}\left\| \left( {{T}_{1}},{{T}_{2}} \right) \right\|_{e}^{2}\le \frac{1}{2}\left\| {{T}_{1}}T_{1}^{*}+{{T}_{2}}T_{2}^{*} \right\|.
\end{equation*}
Now letting ${{T}_{1}}=T$ and ${{T}_{2}}={{T}^{*}}$. We get 
\[\begin{aligned}
   {{\omega }^{2}}\left( T \right)&={{\omega }^{2}}\left( \left[ \begin{matrix}
   O & T  \\
   T & O  \\
\end{matrix} \right] \right) \\ 
 & \le \frac{1}{2}\left\| \left( T,{{T}^{*}} \right) \right\|_{e}^{2} \\ 
 & \le \frac{1}{2}\left\| T{{T}^{*}}+{{T}^{*}}T \right\|, 
\end{aligned}\]
as required.
\end{proof}

\begin{remark}
Among the sharpest upper bounds for $\omega \left( \left[ \begin{matrix}
   O & {{T}_{1}}  \\
   T_{2}^{*} & O  \\
\end{matrix} \right] \right)$ is $\frac{\|T_1\|+\|T_2\|}{2}$. Indeed, this can be seen easily from Lemma \ref{lem_hirz} and the triangle inequality for the usual operator norm.

In this remark, we give an example to support the significance of the bound we found in Corollary \ref{02}. Indeed, if we let 
\[T_1=\left[
\begin{array}{cc}
 4 & 5 \\
 2 & 0 \\
\end{array}
\right], T_2=\left[
\begin{array}{cc}
 4 & 4 \\
 5 & 5 \\
\end{array}
\right],\]
then
\[\left( \frac{\|T_1\|+\|T_2\|}{2}    \right)^2\approx 60.7349\;{\text{and}}\;\frac{1}{2}\left\| T{{T}^{*}}+{{T}^{*}}T \right\|\approx 56.2155.\]
This shows that the bound found in Corollary \ref{02} is significant. However, if we let
\[T_1=\left[
\begin{array}{cc}
 4 & 2 \\
 0 & 4 \\
\end{array}
\right], T_2=\left[
\begin{array}{cc}
 3 & 1 \\
 1 & 0 \\
\end{array}
\right],\]
we have
\[\left( \frac{\|T_1\|+\|T_2\|}{2}    \right)^2\approx 17.7489\;{\text{and}}\;\frac{1}{2}\left\| T{{T}^{*}}+{{T}^{*}}T \right\|\approx 18.1385.\]
These two examples show that neither $\left( \frac{\|T_1\|+\|T_2\|}{2}    \right)^2$ nor 
$\frac{1}{2}\left\| T{{T}^{*}}+{{T}^{*}}T \right\|$ is uniformly better than the other, as an upper bound of $\omega^2 \left( \left[ \begin{matrix}
   O & {{T}_{1}}  \\
   T_{2}^{*} & O  \\
\end{matrix} \right] \right).$
\end{remark}

A well-known inequality for the numerical radius of the product of two operators has been shown in \cite[Remark 1]{Ki1}, where for $A,B\in\mathbb{B}(\mathcal{H})$, one has
	\begin{equation}\label{eq_kitt3}
		\omega \left( {{B}^{*}}A \right)\le\frac{1}{2}\left\| {{\left| A \right|}^{2}}+{{\left| B \right|}^{2}} \right\|.
	\end{equation}

	Interestingly, Theorem \ref{0} implies the following refinement of this inequality via $\|\cdot\|_e$.
\begin{corollary}
Let $A,B\in \mathbb B\left( \mathcal H \right)$. Then
\[\omega \left( {{B}^{*}}A \right)\le \frac{1}{2}\left\| \left( A,B \right) \right\|_{e}^{2}\le \frac{1}{2}\left\| {{\left| A \right|}^{2}}+{{\left| B \right|}^{2}} \right\|.\]
\end{corollary}	
\begin{proof}
If $T_1,T_2\in\mathbb{B}(\mathcal{H}),$ Lemma \ref{lem_prod} implies
\[\max \left\{ \omega \left( T_{2}^{*}{{T}_{1}} \right),\omega \left( {{T}_{1}}T_{2}^{*} \right) \right\}\le {{\omega }^{2}}\left( \left[ \begin{matrix}
   O & {{T}_{1}}  \\
   T_{2}^{*} & O  \\
\end{matrix} \right] \right).\]
Thus, from \eqref{001}, we obtain	
\[\max \left\{ \omega \left( T_{2}^{*}{{T}_{1}} \right),\omega \left( {{T}_{1}}T_{2}^{*} \right) \right\}\le \frac{1}{2}\left\| \left( {{T}_{1}},{{T}_{2}} \right) \right\|_{e}^{2}\le \frac{1}{2}\left\| {{T}_{1}}T_{1}^{*}+{{T}_{2}}T_{2}^{*} \right\|.\]
In particular, if we set ${{T}_{1}}={{B}^{*}}$, ${{T}_{2}}={{A}^{*}}$, and use the fact that
\[{{\left\| \left( {{T}_{1}},\ldots ,{{T}_{n}} \right) \right\|}_{e}}={{\left\| \left( T_{1}^{*},\ldots ,T_{n}^{*} \right) \right\|}_{e}},\]	
we infer that
\[\omega \left( {{B}^{*}}A \right)\le \frac{1}{2}\left\| \left( B,A \right) \right\|_{e}^{2}\le \frac{1}{2}\left\| {{\left| A \right|}^{2}}+{{\left| B \right|}^{2}} \right\|,\]
as required.
\end{proof}

Due to the role of operator matrices in studying the numerical radius, we present the following important bound.
\begin{corollary}\label{04}
Let $T\in \mathbb B\left( \mathcal H \right)$. Then
\[\omega \left( \left[ \begin{matrix}
   O & \Re T  \\
   \Im T & O  \\
\end{matrix} \right] \right)\le \frac{1}{2}{{\left\| {{T}^{*}}T+T{{T}^{*}} \right\|}^{\frac{1}{2}}},\]
where $\Re T$ and $\Im T$ denote the real and imaginary parts of $T$, respectively.
\end{corollary}
\begin{proof}
We have
	\[\begin{aligned}
   \omega \left( \left[ \begin{matrix}
   O & \Re T  \\
   \Im T & O  \\
\end{matrix} \right] \right)&\le \frac{1}{\sqrt{2}}{{\left\| \left( \Re T,\Im T \right) \right\|}_{e}} \quad \text{(by Theorem \ref{0})}\\ 
 & \le \frac{1}{\sqrt{2}}{{\left\| {{\left( \Re T \right)}^{2}}+{{\left( \Im T \right)}^{2}} \right\|}^{\frac{1}{2}}} \quad \text{(by \eqref{001})}\\ 
 & =\frac{1}{\sqrt{2}}{{\left\| \frac{{{T}^{*}}T+T{{T}^{*}}}{2} \right\|}^{\frac{1}{2}}} \\ 
 & =\frac{1}{2}{{\left\| {{T}^{*}}T+T{{T}^{*}} \right\|}^{\frac{1}{2}}},  
\end{aligned}\]
as required.
\end{proof}

The significance of Corollary \ref{04} is explained in the following remark.
\begin{remark}
It has been shown in \cite[Corollary 2.5]{1} that
\begin{equation}\label{03}
\omega \left( \left[ \begin{matrix}
   O & \Re T  \\
   \Im T & O  \\
\end{matrix} \right] \right)\le \omega \left( T \right).
\end{equation}
It follows from Corollary \ref{04} that
	\[{{\omega }^{2}}\left( \left[ \begin{matrix}
   O & \Re T  \\
   \Im T & O  \\
\end{matrix} \right] \right)\le \frac{1}{4}\left\| {{T}^{*}}T+T{{T}^{*}} \right\|\le {{\omega }^{2}}\left( T \right)\]
where we have used \eqref{eq_kitt_intro2} to obtain the second inequality in the above argument. Accordingly, Corollary \ref{04} provides an improvement of \eqref{03}.
\end{remark}

On the other hand, a lower bound for the numerical radius of an operator matrix may be stated in terms of $\|\cdot\|_e$ as follows.
\begin{theorem}\label{08}
Let ${{T}_{1}},{{T}_{2}}\in \mathbb B\left( \mathcal H \right)$. Then
\[\frac{1}{2}\;{{\left\| \left( {{T}_{1}},{{T}_{2}} \right) \right\|}_{e}}\le \omega \left( \left[ \begin{matrix}
   O & {{T}_{1}}  \\
   T_{2}^{*} & O  \\
\end{matrix} \right] \right).\]
\end{theorem}
\begin{proof}
The classical parallelogram law for complex numbers asserts that
	\[{{\left| a \right|}^{2}}+{{\left| b \right|}^{2}}=\frac{{{\left| a+b \right|}^{2}}+{{\left| a-b \right|}^{2}}}{2};\text{ }\left( a,b\in \mathbb{C} \right).\]
From the above identity, one can get
\begin{equation}\label{05}
{{\left| \left\langle {{T}_{1}}x,y \right\rangle  \right|}^{2}}+{{\left| \left\langle {{T}_{2}}x,y \right\rangle  \right|}^{2}}=\frac{1}{2}\left( {{\left| \left\langle \left( {{T}_{1}}+{{T}_{2}} \right)x,y \right\rangle  \right|}^{2}}+{{\left| \left\langle \left( {{T}_{1}}-{{T}_{2}} \right)x,y \right\rangle  \right|}^{2}} \right)
\end{equation}
for any $x,y\in \mathcal H$. Hence,
\begin{equation}\label{00}
{{\left\| \left( {{T}_{1}},{{T}_{2}} \right) \right\|}_{e}}\le \sqrt{\frac{{{\left\| {{T}_{1}}+{{T}_{2}} \right\|}^{2}}+{{\left\| {{T}_{1}}-{{T}_{2}} \right\|}^{2}}}{2}}.
\end{equation}
Now, replacing ${{T}_{2}}$ by ${{e}^{{\rm i}\theta }}{{T}_{2}}$, in \eqref{00}, we get
\begin{equation}\label{00000}
{{\left\| \left( {{T}_{1}},{{T}_{2}} \right) \right\|}_{e}}\le \sqrt{\frac{{{\left\| {{T}_{1}}+{{e}^{{\rm i}\theta }}{{T}_{2}} \right\|}^{2}}+{{\left\| {{T}_{1}}-{{e}^{{\rm i}\theta }}{{T}_{2}} \right\|}^{2}}}{2}}.
\end{equation}
Notice that, by Lemma \ref{lem_hirz},
\begin{equation}\label{000}
\frac{1}{2}\left\| {{T}_{1}}+{{e}^{{\rm i}\theta }}{{T}_{2}} \right\|\le \omega \left( \left[ \begin{matrix}
   O & {{T}_{1}}  \\
   T_{2}^{*} & O  \\
\end{matrix} \right] \right).
\end{equation}
On the other hand, by replacing ${{T}_{2}}$ by $-{{T}_{2}}$, in \eqref{000}, we obtain
\begin{equation}\label{0000}
\frac{1}{2}\left\| {{T}_{1}}-{{e}^{{\rm i}\theta }}{{T}_{2}} \right\|\le \omega \left( \left[ \begin{matrix}
   O & {{T}_{1}}  \\
   -T_{2}^{*} & O  \\
\end{matrix} \right] \right)=\omega \left( \left[ \begin{matrix}
   O & {{T}_{1}}  \\
   T_{2}^{*} & O  \\
\end{matrix} \right] \right).
\end{equation}
Now, combining the relations \eqref{000} and \eqref{0000} with \eqref{00000}, implies
	\[{{\left\| \left( {{T}_{1}},{{T}_{2}} \right) \right\|}_{e}}\le 2\omega \left( \left[ \begin{matrix}
   O & {{T}_{1}}  \\
   T_{2}^{*} & O  \\
\end{matrix} \right] \right),\]
as required.
\end{proof}

Related to the bounds we found in Theorems \ref{0} and \ref{08}, we have the following.
\begin{corollary}\label{cor_blocks}
Let ${{T}_{1}},{{T}_{2}}\in \mathbb B\left( \mathcal H \right)$. Then
\[\omega \left( \left[ \begin{matrix}
   O & \frac{{{T}_{1}}+{{T}_{2}}}{2}  \\
   \frac{T_{1}^{*}-T_{2}^{*}}{2} & O  \\
\end{matrix} \right] \right)\le \frac{1}{2}{{\left\| \left( {{T}_{1}},{{T}_{2}} \right) \right\|}_{e}}\le \omega \left( \left[ \begin{matrix}
   O & {{T}_{1}}  \\
   T_{2}^{*} & O  \\
\end{matrix} \right] \right),\]
and
\[\frac{1}{2}\omega \left( \left[ \begin{matrix}
   O & {{T}_{1}}  \\
   T_{2}^{*} & O  \\
\end{matrix} \right] \right)\le \frac{\sqrt{2}}{4}{{\left\| \left( {{T}_{1}},{{T}_{2}} \right) \right\|}_{e}}\le \omega \left( \left[ \begin{matrix}
   O & \frac{{{T}_{1}}+{{T}_{2}}}{2}  \\
   \frac{T_{1}^{*}-T_{2}^{*}}{2} & O  \\
\end{matrix} \right] \right).\]
\end{corollary}
\begin{proof}
It follows from \eqref{05} that
\begin{equation}\label{07}
\frac{\sqrt{2}}{2}{{\left\| \left( {{T}_{1}}+{{T}_{2}},{{T}_{1}}-{{T}_{2}} \right) \right\|}_{e}}={{\left\| \left( {{T}_{1}},{{T}_{2}} \right) \right\|}_{e}}.
\end{equation}
Therefore,
\[\begin{aligned}
   \omega \left( \left[ \begin{matrix}
   O & {{T}_{1}}+{{T}_{2}}  \\
   {{\left( {{T}_{1}}-{{T}_{2}} \right)}^{*}} & O  \\
\end{matrix} \right] \right)&\le \frac{\sqrt{2}}{2}{{\left\| \left( {{T}_{1}}+{{T}_{2}},{{T}_{1}}-{{T}_{2}} \right) \right\|}_{e}} \quad \text{(by Theorem \ref{0})}\\ 
 & ={{\left\| \left( {{T}_{1}},{{T}_{2}} \right) \right\|}_{e}} \quad \text{(by \eqref{07})}\\ 
 & \le 2\omega \left( \left[ \begin{matrix}
   O & {{T}_{1}}  \\
   T_{2}^{*} & O  \\
\end{matrix} \right] \right) \quad \text{(by Theorem \ref{08})}. 
\end{aligned}\]
This proves the first inequality. To prove the second inequality, notice that
\[\begin{aligned}
   \sqrt{2}\omega \left( \left[ \begin{matrix}
   O & {{T}_{1}}  \\
   T_{2}^{*} & O  \\
\end{matrix} \right] \right)&\le {{\left\| \left( {{T}_{1}},{{T}_{2}} \right) \right\|}_{e}} \quad \text{(by Theorem \ref{0})}\\ 
 & =\frac{\sqrt{2}}{2}{{\left\| \left( {{T}_{1}}+{{T}_{2}},{{T}_{1}}-{{T}_{2}} \right) \right\|}_{e}} \quad \text{(by \eqref{07})}\\ 
 & \le \sqrt{2}\omega \left( \left[ \begin{matrix}
   O & {{T}_{1}}+{{T}_{2}}  \\
   {{\left( {{T}_{1}}-{{T}_{2}} \right)}^{*}} & O  \\
\end{matrix} \right] \right)  \quad \text{(by Theorem \ref{08})}.
\end{aligned}\]
This completes the proof.
\end{proof}
\begin{remark}
Let $T\in\mathbb{B}(\mathcal{H}).$ If we let $T_1=T_2=T$ in Corollary \ref{cor_blocks}, we obtain
\[\omega \left( \left[ \begin{matrix}
   O & T  \\
   O & O  \\
\end{matrix} \right] \right)\leq \frac{1}{2}\|(T,T)\|_e\leq \omega \left( \left[ \begin{matrix}
   O & T  \\
   T^{*} & O  \\
\end{matrix} \right] \right), \]
and
\[\frac{1}{2}\omega \left( \left[ \begin{matrix}
   O & T  \\
   T^{*} & O  \\
\end{matrix} \right] \right)\leq \frac{\sqrt{2}}{4}\|(T,T)\|_e\leq \omega \left( \left[ \begin{matrix}
   O & T  \\
  O & O  \\
\end{matrix} \right] \right).\]
This implies that
\[\frac{1}{2}\|T\|\leq \frac{1}{2}\|(T,T)\|_e\leq \|T\|,\]
and 
\[\frac{1}{2}\|T\|\leq \frac{\sqrt{2}}{4}\|(T,T)\|_e\leq \frac{1}{2}\|T\|.\]
From this, we infer that for any $T\in\mathbb{B}(\mathcal{H})$, one has the identity $\sqrt{2}\;\|T\|=\|(T,T)\|_e.$

Notice that, from the definition of $\|\cdot\|_e$, this implies that subject to the constraint $\lambda_1,\lambda_2\in\mathbb{C}$ are such that $|\lambda_1|^2+|\lambda_2|^2=1,$ the maximum value of $|\lambda_1+\lambda_2|$ is $\sqrt{2}.$

\end{remark}

From \cite[Corollary 2.5]{1}, we know that
 \[\frac{1}{2}\omega \left( T \right)\le \omega \left( \left[ \begin{matrix}
   O & \Re T  \\
   \Im T & O  \\
\end{matrix} \right] \right)\le \omega \left( T \right).\]
For $T\in\mathbb{B}(\mathcal{H})$, if we let $T_1=T^*$ and $T_2=T$ in Corollary \ref{cor_blocks}, we reach the following form, by noting that $\omega\left(\left[\begin{matrix} O&T_1\\e^{\textup{i}\theta}T_2&O\end{matrix}\right]\right)=\omega\left(\left[\begin{matrix} O&T_1\\T_2&O\end{matrix}\right]\right).$
\begin{corollary}\label{cor_blockss}
Let $T\in\mathbb{B}(\mathcal{H}).$ Then
\[\omega\left(\left[\begin{matrix} O&\Re T\\ \ \Im T&O\end{matrix}\right]\right)\leq \frac{1}{2}\left\|(T^*,T)\right\|_e\leq \omega(T),\]
and
\[\frac{1}{2}\omega(T)\leq  \frac{\sqrt{2}}{4}\left\|(T^*,T)\right\|_e\leq \omega\left(\left[\begin{matrix} O&\Re T\\  \Im T&O\end{matrix}\right]\right).\]
\end{corollary}

Another upper bound for $\|(T_1, T_2)\|_e$ with a useful application can be stated as follows. The significance of this estimate can be seen in the consequent corollary and remark.
\begin{theorem}\label{thm_i}
Let ${{T}_{1}},{{T}_{2}}\in \mathbb B\left( \mathcal H \right)$. Then
\[{{\left\| \left( {{T}_{1}},{{T}_{2}} \right) \right\|}_{e}}\le \sqrt{\omega \left( \left| {{T}_{1}} \right|+{\rm i}\left| {{T}_{2}} \right| \right)\omega \left( \left| T_{1}^{*} \right|+{\rm i}\left| T_{2}^{*} \right| \right)}.\]
\end{theorem}
\begin{proof}
Let $x,y \in \mathcal H$ with $\left\| x \right\|=\left\| y \right\|=1$. Then
\[\begin{aligned}
   {{\left| \left\langle {{T}_{1}}x,y \right\rangle  \right|}^{2}}+{{\left| \left\langle {{T}_{2}}x,y \right\rangle  \right|}^{2}}&\le \left\langle \left| {{T}_{1}} \right|x,x \right\rangle \left\langle \left| T_{1}^{*} \right|y,y \right\rangle +\left\langle \left| {{T}_{2}} \right|x,x \right\rangle \left\langle \left| T_{2}^{*} \right|y,y \right\rangle  \\ 
   &\quad \text{(by the mixed Schwarz inequality)}\\
 & \le \sqrt{{{\left\langle \left| {{T}_{1}} \right|x,x \right\rangle }^{2}}+{{\left\langle \left| {{T}_{2}} \right|x,x \right\rangle }^{2}}}\sqrt{{{\left\langle \left| T_{1}^{*} \right|y,y \right\rangle }^{2}}+{{\left\langle \left| T_{2}^{*} \right|y,y \right\rangle }^{2}}} \\
 &\quad \text{(by the Cauchy-Schwarz inequality)}\\
 & =\left| \left\langle \left| {{T}_{1}} \right|x,x \right\rangle +{\rm i}\left\langle \left| {{T}_{2}} \right|x,x \right\rangle  \right|\left| \left\langle \left| T_{1}^{*} \right|y,y \right\rangle +{\rm i}\left\langle \left| T_{2}^{*} \right|y,y \right\rangle  \right| \\ 
 & =\left| \left\langle \left( \left| {{T}_{1}} \right|+{\rm i}\left| {{T}_{2}} \right| \right)x,x \right\rangle  \right|\left| \left\langle \left( \left| T_{1}^{*} \right|+{\rm i}\left| T_{2}^{*} \right| \right)y,y \right\rangle  \right| \\ 
 & \le \omega \left( \left| {{T}_{1}} \right|+{\rm i}\left| {{T}_{2}} \right| \right)\omega \left( \left| T_{1}^{*} \right|+{\rm i}\left| T_{2}^{*} \right| \right).  
\end{aligned}\]
That is
\[\sqrt{{{\left| \left\langle {{T}_{1}}x,y \right\rangle  \right|}^{2}}+{{\left| \left\langle {{T}_{2}}x,y \right\rangle  \right|}^{2}}}\le \sqrt{\omega \left( \left| {{T}_{1}} \right|+{\rm i}\left| {{T}_{2}} \right| \right)\omega \left( \left| T_{1}^{*} \right|+{\rm i}\left| T_{2}^{*} \right| \right)}.\]
By taking supremum over $x,y \in \mathcal H$ with $\left\| x \right\|=\left\| y \right\|=1$, we get the desired inequality.
\end{proof}

Combining Theorems \ref{0} and \ref{thm_i} implies the following new upper bound for the numerical radius of the off-diagonal operator matrix.

\begin{corollary}\label{cor_new_block}
Let $T_1,T_2\in\mathbb{B}(\mathcal{H})$. Then
\[\omega \left( \left[ \begin{matrix}
   O & {{T}_{1}}  \\
   T_{2}^{*} & O  \\
\end{matrix} \right] \right)\le \frac{1}{\sqrt{2}}\sqrt{\omega \left( \left| {{T}_{1}} \right|+{\rm i}\left| {{T}_{2}} \right| \right)\omega \left( \left| T_{1}^{*} \right|+{\rm i}\left| T_{2}^{*} \right| \right)}.\]
\end{corollary}

\begin{remark}
It has been shown in \cite{1} that
\begin{equation*}
\omega \left( \left[ \begin{matrix}
   O & {{T}_{1}}  \\
   T_{2}^{*} & O  \\
\end{matrix} \right] \right)\le \frac{\omega(T_1+T_2^*)+\omega(T_1-T_2^*)}{2},
\end{equation*}
and
\begin{equation}\label{eq_block_normsover2}
\omega \left( \left[ \begin{matrix}
   O & {{T}_{1}}  \\
   T_{2}^{*} & O  \\
\end{matrix} \right] \right)\le \frac{\|T_1\|+\|T_2\|}{2}.
\end{equation}
Here, we give examples to show that the bound we found in Corollary \ref{cor_new_block} can be sharper than these two celebrated bounds.

For example, if we let
\[T_1=\left[
\begin{array}{cc}
 2 & 4 \\
 3 & 4 \\
\end{array}
\right], T_2=\left[
\begin{array}{cc}
 5 & 3 \\
 3 & 4 \\
\end{array}
\right],\] 
then
\[\omega \left( \left[ \begin{matrix}
   O & {{T}_{1}}  \\
   T_{2}^{*} & O  \\
\end{matrix} \right] \right)\approx 7.01793, \frac{\|T_1\|+\|T_2\|}{2} \approx 7.11141, \frac{\omega(T_1+T_2^*)+\omega(T_1-T_2^*)}{2}\approx 8.55017,\]
and 
\[\frac{1}{\sqrt{2}}\sqrt{\omega \left( \left| {{T}_{1}} \right|+{\rm i}\left| {{T}_{2}} \right| \right)\omega \left( \left| T_{1}^{*} \right|+{\rm i}\left| T_{2}^{*} \right| \right)}\approx 7.04011.\]

However, there are other examples where the bound in \eqref{eq_block_normsover2} is better than that in Corollary \ref{cor_new_block}. Indeed, if we let 
\[T_1=\left[
\begin{array}{cc}
 5 & 3 \\
 2 & 2 \\
\end{array}
\right], T_2=\left[
\begin{array}{cc}
 4 & 5 \\
 5 & 0 \\
\end{array}
\right],\]
then numerical calculations show that
\[\omega \left( \left[ \begin{matrix}
   O & {{T}_{1}}  \\
   T_{2}^{*} & O  \\
\end{matrix} \right] \right)\approx 6.89958, \frac{\|T_1\|+\|T_2\|}{2} \approx 6.91809, \frac{\omega(T_1+T_2^*)+\omega(T_1-T_2^*)}{2}\approx 8.91299,\]
and 
\[\frac{1}{\sqrt{2}}\sqrt{\omega \left( \left| {{T}_{1}} \right|+{\rm i}\left| {{T}_{2}} \right| \right)\omega \left( \left| T_{1}^{*} \right|+{\rm i}\left| T_{2}^{*} \right| \right)}\approx 6.92022.\]
\end{remark}

In the above discussion, we provided several results about the possible relations between the numerical radius of the off-diagonal operator matrix and the Euclidean operator norm of its off-diagonal entries. In the next result, we find a relation with the Euclidean operator norm.
\begin{theorem}\label{thm_block_we}
Let $T_1,T_2\in\mathbb{B}(\mathcal{H})$. Then
\[\omega \left( \left[ \begin{matrix}
   O & {{T}_{1}}  \\
   T_{2}^{*} & O  \\
\end{matrix} \right] \right)\le \frac{\sqrt{2}}{2}\sqrt{{{\omega }_{e}}\left( \left| {{T}_{1}} \right|,\left| {{T}_{2}} \right| \right){{\omega }_{e}}\left( \left| T_{1}^{*} \right|,\left| T_{2}^{*} \right| \right)}.\]
\end{theorem}
\begin{proof}
Let $x,y \in \mathcal H$. Then by \eqref{1}, we have
\[\begin{aligned}
   {{\left| \left\langle \left( {{T}_{1}}+{{T}_{2}} \right)x,y \right\rangle  \right|}^{2}}&\le 2\left( {{\left| \left\langle {{T}_{1}}x,y \right\rangle  \right|}^{2}}+{{\left| \left\langle {{T}_{2}}x,y \right\rangle  \right|}^{2}} \right) \\ 
 & \le 2\left( \left\langle \left| {{T}_{1}} \right|x,x \right\rangle \left\langle \left| T_{1}^{*} \right|y,y \right\rangle +\left\langle \left| {{T}_{2}} \right|x,x \right\rangle \left\langle \left| T_{2}^{*} \right|y,y \right\rangle  \right) \\ 
 &\quad \text{(by the mixed Schwarz inequality)}\\
 & \le 2\sqrt{{{\left\langle \left| {{T}_{1}} \right|x,x \right\rangle }^{2}}+{{\left\langle \left| {{T}_{2}} \right|x,x \right\rangle }^{2}}}\sqrt{{{\left\langle \left| T_{1}^{*} \right|y,y \right\rangle }^{2}}+{{\left\langle \left| T_{2}^{*} \right|y,y \right\rangle }^{2}}} \\ 
 &\quad \text{(by the Cauchy-Schwarz inequality)}\\
 & \le 2\;{{\omega }_{e}}\left( \left| {{T}_{1}} \right|,\left| {{T}_{2}} \right| \right){{\omega }_{e}}\left( \left| T_{1}^{*} \right|,\left| T_{2}^{*} \right| \right).
\end{aligned}\]
That is,
\[\left| \left\langle \left( {{T}_{1}}+{{T}_{2}} \right)x,y \right\rangle  \right|\le \sqrt{2}\sqrt{{{\omega }_{e}}\left( \left| {{T}_{1}} \right|,\left| {{T}_{2}} \right| \right){{\omega }_{e}}\left( \left| T_{1}^{*} \right|,\left| T_{2}^{*} \right| \right)}.\]
By taking supremum over unit vectors $x,y \in \mathcal H$, we get
\[\left\| {{T}_{1}}+{{T}_{2}} \right\|\le \sqrt{2}\sqrt{{{\omega }_{e}}\left( \left| {{T}_{1}} \right|,\left| {{T}_{2}} \right| \right){{\omega }_{e}}\left( \left| T_{1}^{*} \right|,\left| T_{2}^{*} \right| \right)}.\]
Replacing ${{T}_{2}}$ by ${{e}^{{\rm i}\theta }}{{T}_{2}}$ produces
\[\frac{1}{2}\left\| {{T}_{1}}+{{e}^{{\rm i}\theta }}{{T}_{2}} \right\|\le \frac{\sqrt{2}}{2}\sqrt{{{\omega }_{e}}\left( \left| {{T}_{1}} \right|,\left| {{T}_{2}} \right| \right){{\omega }_{e}}\left( \left| T_{1}^{*} \right|,\left| T_{2}^{*} \right| \right)}.\]
Now, the result follows by taking supremum over $\theta \in \mathbb{R}$.
\end{proof}
\begin{remark}
It follows from the definition of $\omega_e$ that if $T=\Re T+{\textup{i}}\Im T$ is the Cartesian decomposition of $T$, then $\omega_e(\Re T,\Im T)=\omega(T).$ So, if $T$ is accretive-dissipative (that is $\Re T,\Im T\geq O$), then Theorem \ref{thm_block_we} implies
\[\omega\left(\left[\begin{matrix}O&\Re T\\ \Im T&O\end{matrix}\right]\right)\leq \frac{1}{\sqrt{2}}\omega(T).\]
Due to Lemma \ref{lem_hirz}, we have, for any two operators $T_1,T_2\in\mathbb{B}(\mathcal{H})$,
\[\frac{1}{2}\|T_1+{\textup{i}}T_2\|\leq \omega\left(\left[\begin{matrix}O&T_1\\T_2^*&O\end{matrix}\right]\right).\]
Combining the above two inequalities, we infer that if $T$ is accretive-dissipative, then
\[\|T\|\leq 2\omega\left(\left[\begin{matrix}O&\Re T\\ \Im T&O\end{matrix}\right]\right)\leq \sqrt{2}\omega(T).\]
Notice that this provides a refinement of the first inequality in \eqref{eq_equiv_intro} when $T$ is accretive-dissipative. 

We point out that the inequality $\|T\|\leq\sqrt{2}\;\omega(T)$ has been known for accretive-dissipative operators $T$, as one can see in \cite{MS}.
\end{remark}

\section{The generalized Euclidean operator radius}
In this section, we explore the $f$-operator radius. In particular, upper bounds for $\omega_f$ will be presented in a way that extends some bounds for $\omega.$

As applications, we will be able to present new refinements and generalizations of \eqref{eq_kitt_intro2} and \eqref{eq_kitt3}.


\begin{theorem}
	\label{thm2.2}	For $j=1,\ldots,n,$ let $T_j\in\mathbb{B}(\mathcal{H})$, and let $f:[0,\infty)\to [0,\infty)$ be an increasing operator convex function such that $f(0)=0$. If $0\leq\alpha,\lambda\leq 1,$ then
	\begin{align}
&\omega _f \left( {T_1 , \ldots ,T_n } \right) 
\nonumber\\ 
&\le  f^{ - 1} \left( {\left\| {\sum\limits_{j = 1}^n {\left[ {\left( {1 - \lambda } \right)f\left( {\frac{{\left( {1 - \lambda } \right)\left| {T_j } \right|^{2\alpha }  + \left( {1 + \lambda } \right)\left| {T_j^* } \right|^{2\left( {1 - \alpha } \right)} }}{2}} \right)} \right.} } \right.} \right. 
\nonumber\\ 
&\qquad\qquad\qquad \left. {\left. {\left. { + \lambda f\left( {\frac{{\left( {2 - \lambda } \right)\left| {T_j } \right|^{2\alpha }  + \lambda \left| {T_j^* } \right|^{2\left( {1 - \alpha } \right)} }}{2}} \right)} \right]} \right\|} \right) 
\nonumber\\ 
&\le f^{ - 1} \left( {\left\| {\int_0^1 {\sum\limits_{j = 1}^n {f\left( {t\left| {T_j } \right|^{2\alpha }  + \left( {1 - t} \right)\left| {T_j^* } \right|^{2\left( {1 - \alpha } \right)} } \right)dt} } } \right\|} \right) 
\nonumber\\ 
&\le f^{ - 1} \left( {\frac{1}{2}\left\| {\sum\limits_{j = 1}^n {\left( {f\left( {\left| {T_j } \right|^{2\alpha } } \right) + f\left( {\left| {T_j^* } \right|^{2\left( {1 - \alpha } \right)} } \right)} \right)} } \right\|} \right). \nonumber
\end{align}
\end{theorem}
\begin{proof}
 Let $x \in \mathscr{H}$ be a unit vector.  Noting Lemma \ref{lemma_mixed}, the fact that $f$ is an increasing convex function, and the simple equation $\frac{(1-\lambda)^2+\lambda(2-\lambda)}{2}a+\frac{(1-\lambda^2)+\lambda^2}{2}b=\frac{a+b}{2},$ we have
\begin{align*}
f\left( {\left| {\left\langle {T_j x,x} \right\rangle } \right|} \right) &\le f\left( {\sqrt {\left\langle {\left| {T_j } \right|^{2\alpha } x,x} \right\rangle \left\langle {\left| {T_j^* } \right|^{2\left( {1 - \alpha } \right)} x,x} \right\rangle } } \right) 
\\ 
&\le \left( {1 - \lambda } \right) {\left\langle {f\left( {\frac{{\left( {1 - \lambda } \right)\left| {T_j } \right|^{2\alpha }  + \left( {1 + \lambda } \right)\left| {T_j^* } \right|^{2\left( {1 - \alpha } \right)} }}{2}} \right)x,x} \right\rangle } 
\\ 
&\qquad+ \lambda  {\left\langle {f\left( {\frac{{\left( {2 - \lambda } \right)\left| {T_j } \right|^{2\alpha }  + \lambda \left| {T_j^* } \right|^{2\left( {1 - \alpha } \right)} }}{2}} \right)x,x} \right\rangle } 
\\ 
&= 	\left\langle {\left[ {\left( {1 - \lambda } \right)f\left( {\frac{{\left( {1 - \lambda } \right)\left| {T_j } \right|^{2\alpha }  + \left( {1 + \lambda } \right)\left| {T_j^* } \right|^{2\left( {1 - \alpha } \right)} }}{2}} \right)} \right.} \right. 
\\ 
&\qquad
\left. {\left. { + \lambda f\left( {\frac{{\left( {2 - \lambda } \right)\left| {T_j } \right|^{2\alpha }  + \lambda \left| {T_j^* } \right|^{2\left( {1 - \alpha } \right)} }}{2}} \right)} \right]x,x} \right\rangle.
\end{align*}
Now, let $g(t)=\left<f\left( {t\left| {T_j } \right|^{2\alpha }  + \left( {1 - t} \right)\left| {T_j^* } \right|^{2\left( {1 - \alpha } \right)} } \right)x,x\right>.$ Operator convexity of $f$ implies convexity of $g$.

Since $f$ is operator convex, the operator version of Lemma \ref{lem_farissi} implies that
\begin{align*}
&\left\langle {\left[ {\left( {1 - \lambda } \right)f\left( {\frac{{\left( {1 - \lambda } \right)\left| {T_j } \right|^{2\alpha }  + \left( {1 + \lambda } \right)\left| {T_j^* } \right|^{2\left( {1 - \alpha } \right)} }}{2}} \right)} \right.} \right. 
\\ 
&\qquad
\left. {\left. { + \lambda f\left( {\frac{{\left( {2 - \lambda } \right)\left| {T_j } \right|^{2\alpha }  + \lambda \left| {T_j^* } \right|^{2\left( {1 - \alpha } \right)} }}{2}} \right)} \right]x,x} \right\rangle\\ 
&\le \left\langle {\left( {\int_0^1 {f\left( {t\left| {T_j } \right|^{2\alpha }  + \left( {1 - t} \right)\left| {T_j^* } \right|^{2\left( {1 - \alpha } \right)} } \right)dt} } \right)x,x} \right\rangle  
\\ 
&\le \left\langle {\left( {\frac{{f\left( {\left| {T_j } \right|^{2\alpha } } \right) + f\left( {\left| {T_j^* } \right|^{2\left( {1 - \alpha } \right)} } \right)}}{2}} \right)x,x} \right\rangle. 
\end{align*}
Thus, we have shown that
\begin{align*}
&\sum\limits_{j = 1}^n {f\left( {\left| {\left\langle {T_j x,x} \right\rangle } \right|} \right)}  
\\ 
&\le \sum\limits_{j = 1}^n {\left\langle {\left[ {\left( {1 - \lambda } \right)f\left( {\frac{{\left( {1 - \lambda } \right)\left| {T_j } \right|^{2\alpha }  + \left( {1 + \lambda } \right)\left| {T_j^* } \right|^{2\left( {1 - \alpha } \right)} }}{2}} \right)} \right.} \right.}  
\\ 
&\qquad\left. {\left. { + \lambda f\left( {\frac{{\left( {2 - \lambda } \right)\left| {T_j } \right|^{2\alpha }  + \lambda \left| {T_j^* } \right|^{2\left( {1 - \alpha } \right)} }}{2}} \right)} \right]x,x} \right\rangle  
\\ 
&\le \sum\limits_{j = 1}^n {\left\langle {\left( {\int_0^1 {f\left( {t\left| {T_j } \right|^{2\alpha }  + \left( {1 - t} \right)\left| {T_j^* } \right|^{2\left( {1 - \alpha } \right)} } \right)dt} } \right)x,x} \right\rangle }  
\\ 
&\le \sum\limits_{j = 1}^n {\left\langle {\left( {\frac{{f\left( {\left| {T_j } \right|^{2\alpha } } \right) + f\left( {\left| {T_j^* } \right|^{2\left( {1 - \alpha } \right)} } \right)}}{2}} \right)x,x} \right\rangle }, 
\end{align*}
and this is equivalent to 
\begin{align*}
&\sum\limits_{j = 1}^n {f\left( {\left| {\left\langle {T_j x,x} \right\rangle } \right|} \right)}  
\\ 
&\le \left\langle {\left( {\sum\limits_{j = 1}^n {\left[ {\left( {1 - \lambda } \right)f\left( {\frac{{\left( {1 - \lambda } \right)\left| {T_j } \right|^{2\alpha }  + \left( {1 + \lambda } \right)\left| {T_j^* } \right|^{2\left( {1 - \alpha } \right)} }}{2}} \right)} \right.} } \right.} \right. 
\\ 
&\qquad\qquad	\left. {\left. {\left. { + \lambda f\left( {\frac{{\left( {2 - \lambda } \right)\left| {T_j } \right|^{2\alpha }  + \lambda \left| {T_j^* } \right|^{2\left( {1 - \alpha } \right)} }}{2}} \right)} \right]} \right)x,x} \right\rangle   
\\ 
&\le \left\langle {\left( {\int_0^1 {\sum\limits_{j = 1}^n {f\left( {t\left| {T_j } \right|^{2\alpha }  + \left( {1 - t} \right)\left| {T_j^* } \right|^{2\left( {1 - \alpha } \right)} } \right)dt} } } \right)x,x} \right\rangle  
\\ 
&\le \frac{1}{2}\left\langle {\sum\limits_{j = 1}^n {\left( {f\left( {\left| {T_j } \right|^{2\alpha } } \right) + f\left( {\left| {T_j^* } \right|^{2\left( {1 - \alpha } \right)} } \right)} \right)x} ,x} \right\rangle.
\end{align*}
Since $f$ is increasing, we have
\begin{align*}
&f^{ - 1} \left( {\sum\limits_{j = 1}^n {f\left( {\left| {\left\langle {T_j x,x} \right\rangle } \right|} \right)} } \right) \\ 
&\le 
  f^{ - 1} \left( {\left\langle {\left( {\sum\limits_{j = 1}^n {\left[ {\left( {1 - \lambda } \right)f\left( {\frac{{\left( {1 - \lambda } \right)\left| {T_j } \right|^{2\alpha }  + \left( {1 + \lambda } \right)\left| {T_j^* } \right|^{2\left( {1 - \alpha } \right)} }}{2}} \right)} \right.} } \right.} \right.} \right. \\ 
	&\qquad\qquad\qquad 
	\left. {\left. {\left. {\left. { + \lambda f\left( {\frac{{\left( {2 - \lambda } \right)\left| {T_j } \right|^{2\alpha }  + \lambda \left| {T_j^* } \right|^{2\left( {1 - \alpha } \right)} }}{2}} \right)} \right]} \right)x,x} \right\rangle } \right)  
\\ 
&\le f^{ - 1} \left( {\left\langle {\left( {\int_0^1 {\sum\limits_{j = 1}^n {f\left( {t\left| {T_j } \right|^{2\alpha }  + \left( {1 - t} \right)\left| {T_j^* } \right|^{2\left( {1 - \alpha } \right)} } \right)dt} } } \right)x,x} \right\rangle } \right) 
\\ 
&\le f^{ - 1} \left( {\frac{1}{2}\left\langle {\sum\limits_{j = 1}^n {\left( {f\left( {\left| {T_j } \right|^{2\alpha } } \right) + f\left( {\left| {T_j^* } \right|^{2\left( {1 - \alpha } \right)} } \right)} \right)x} ,x} \right\rangle } \right). 
\end{align*}
Again, since $f^{-1}$ is decreasing, we deduce the required result by taking the supremum over all unit vector $x\in \mathscr{H}$. 
\end{proof}


\begin{corollary}
	\label{cor2.2}	Under the assumptions of Theorem \ref{thm2.2}, we have
	\begin{align}
		&\omega _f \left( {T_1 , \ldots ,T_n } \right) 
		\nonumber\\ 
		&\le  f^{ - 1} \left( {\left\| {\sum\limits_{j = 1}^n {\left[ {\left( {1 - \lambda } \right)f\left( {\frac{{\left( {1 - \lambda } \right)\left| {T_j } \right|  + \left( {1 + \lambda } \right)\left| {T_j^* } \right|  }}{2}} \right)} \right.} } \right.} \right. 
		\nonumber\\ 
		&\qquad\qquad\qquad \left. {\left. {\left. { + \lambda f\left( {\frac{{\left( {2 - \lambda } \right)\left| {T_j } \right|   + \lambda \left| {T_j^* } \right|  }}{2}} \right)} \right]} \right\|} \right) 
		\nonumber\\ 
		&\le f^{ - 1} \left( {\left\| {\int_0^1 {\sum\limits_{j = 1}^n {f\left( {t\left| {T_j } \right|  + \left( {1 - t} \right)\left| {T_j^* } \right|  } \right)dt} } } \right\|} \right) 
		\nonumber\\ 
		&\le f^{ - 1} \left( {\frac{1}{2}\left\| {\sum\limits_{j = 1}^n {\left( {f\left( {\left| {T_j } \right|  } \right) + f\left( {\left| {T_j^* } \right|  } \right)} \right)} } \right\|} \right). \nonumber
	\end{align}
In particular,
	\begin{align}
	&\omega _f \left( {T_1 , \ldots ,T_n } \right) 
	\nonumber\\ 
	&\le  	f^{ - 1} \left( {\frac{1}{2}\left\| {\sum\limits_{j = 1}^n {\left[ {f\left( {\frac{{\left| {T_j } \right| + 3\left| {T_j^* } \right|}}{4}} \right) + f\left( {\frac{{3\left| {T_j } \right| + \left| {T_j^* } \right|}}{4}} \right)} \right]} } \right\|} \right)
	\label{eq2.6}\\ 
	&\le f^{ - 1} \left( {\left\| {\int_0^1 {\sum\limits_{j = 1}^n {f\left( {t\left| {T_j } \right|  + \left( {1 - t} \right)\left| {T_j^* } \right|  } \right)dt} } } \right\|} \right) 
	\nonumber\\ 
	&\le f^{ - 1} \left( {\frac{1}{2}\left\| {\sum\limits_{j = 1}^n {\left( {f\left( {\left| {T_j } \right|  } \right) + f\left( {\left| {T_j^* } \right|  } \right)} \right)} } \right\|} \right). \nonumber
\end{align}
\end{corollary}
Since the function $f \left(t\right) = t^p$, $1 \le p \le 2$ is an increasing operator convex function, Theorem \ref{thm2.2} implies the following.


\begin{corollary}
	\label{cor2.3}	Under the assumptions of Theorem \ref{thm2.2}, we have
	\begin{align*}
		&\omega _f \left( {T_1 , \ldots ,T_n } \right) 
		\nonumber\\ 
		&\le   \left( {\left\| {\sum\limits_{j = 1}^n {\left[ {\left( {1 - \lambda } \right)\left( {\frac{{\left( {1 - \lambda } \right)\left| {T_j } \right|  + \left( {1 + \lambda } \right)\left| {T_j^* } \right|  }}{2}} \right)^p} \right.} } \right.} \right. 
		 \\ 
		&\qquad\qquad\qquad \left. {\left. {\left. { + \lambda \left( {\frac{{\left( {2 - \lambda } \right)\left| {T_j } \right|   + \lambda \left| {T_j^* } \right|  }}{2}} \right)^p} \right]} \right\|} \right)^{\frac{1}{p}} 
		\nonumber\\ 
		&\le   \left( {\left\| {\int_0^1 {\sum\limits_{j = 1}^n { \left( {t\left| {T_j } \right|  + \left( {1 - t} \right)\left| {T_j^* } \right|  } \right)^p dt} } } \right\|} \right)^{\frac{1}{p}} 
		\nonumber\\ 
		&\le  \left( {\frac{1}{2}\left\| {\sum\limits_{j = 1}^n {\left( { \left| {T_j } \right|   ^p +   \left| {T_j^* } \right|  ^p} \right)} } \right\|} \right)^{\frac{1}{p}}. \nonumber
	\end{align*}
	\end{corollary}
\begin{remark}
Setting $n=1$ and $T_1=T$ in \eqref{eq2.6}, and using the function $f(t)=t^2$ yield the following refinement of \cite[Corollary 2.1]{SM2021}:
	\begin{align*}
\omega^2 \left( {T} \right) 
	&\le \frac{1}{2}\left\| {  { \left( {\frac{{\left| {T  } \right| + 3\left| {T ^* } \right|}}{4}} \right)^2 +  \left( {\frac{{3\left| {T  } \right| + \left| {T ^* } \right|}}{4}} \right)^2 } } \right\| 
	\\ 
	&\le    \left\| {\int_0^1 {  {\left( {t\left| {T  } \right|  + \left( {1 - t} \right)\left| {T ^* } \right|  } \right)^2dt} } } \right\| 
	\nonumber\\ 
	&\le   \frac{1}{2}\left\| {     \left| {T  } \right|  ^2 +  \left| {T ^* } \right|  ^2   } \right\|.  
\end{align*}
Of course, this provides a two-term refinement of the second inequality in \eqref{eq_kitt_intro2}.
Notice that the function $f(t)=t$ implies \eqref{eq_kitt_intro1}.
\end{remark}

When treating numerical radius bounds, it is interesting to investigate the numerical radius of the product of two operators. The following result presents such a treatment for the $f$-operator radius.
\begin{theorem}\label{thm_squared}
	\label{thm}	For $j=1,\ldots,n,$ let $S_j,T_j\in\mathbb{B}(\mathcal{H})$, and let $f:[0,\infty)\to [0,\infty)$ be an increasing operator convex function such that $f(0)=0$. Then
	\begin{align}
		\omega _f \left( {T^*_1S_1 , \ldots ,T^*_nS_n } \right) &\le f^{ - 1} \left( {\left\| {\int_0^1 {\left( {\sum\limits_{j = 1}^n {f\left( {t\left| {S_j } \right|^2   + \left( {1 - t} \right)\left| {T_j  } \right|^2 } \right)} } \right)dt} } \right\|} \right) 
		\nonumber\\
		&\le f^{ - 1} \left( {\frac{1}{2}\left\| {\sum\limits_{j = 1}^n {f\left( {\left| {S_j } \right|^2  } \right) + f\left( {\left| {T_j  } \right|^2  } \right)} } \right\|} \right)\nonumber.
	\end{align}
\end{theorem}
\begin{proof}
Let $x \in \mathscr{H}$ be a unit vector, then  by the Cauchy-Schwarz inequality, we have 	
	\begin{align}
		f\left( {\left| {\left\langle {T^*_jS_j x,x} \right\rangle } \right|} \right) &=f\left( {\left| {\left\langle {S_j x,T_jx} \right\rangle } \right|} \right) 
		\nonumber\\
		&\le f\left( {\left\| {S_j x} \right\|\left\| {T_j x} \right\|} \right)
		\nonumber\\
		&\le f\left( {\sqrt {\left\langle {\left| {S_j } \right|^2 x,x} \right\rangle \left\langle {\left| {T_j  } \right|^2  x,x} \right\rangle } } \right) 
		\nonumber\\ 
		&\le f\left( {\frac{{\left\langle {\left| {S_j } \right|^2 x,x} \right\rangle  + \left\langle {\left| {T_j  } \right|^2  x,x} \right\rangle }}{2}} \right) 
		\nonumber\\ 
		&\le \int_0^1 {f\left( {t\left\langle {\left| {S_j } \right|^2 x,x} \right\rangle  + \left( {1 - t} \right)\left\langle {\left| {T_j  } \right|^2 x,x} \right\rangle } \right)dt}.  \label{eq2.8}
	\end{align}	
	On the other hand, convexity of $f$ implies
	\begin{align*}
		&f\left( {t\left\langle {\left| {S_j } \right|^2  x,x} \right\rangle  + \left( {1 - t} \right)\left\langle {\left| {T_j } \right|^2  x,x} \right\rangle } \right) 
		\\
		&= f\left( {\left\langle {\left[ {t\left| {S_j } \right|^2  + \left( {1 - t} \right)\left| {T_j  } \right|^2  } \right]x,x} \right\rangle } \right) 
		\\ 
		&\le \left\langle {f\left( {t\left| {S_j } \right|^2 + \left( {1 - t} \right)\left| {T_j  } \right|^2  } \right)x,x} \right\rangle  
		\\ 
		&\le t\left\langle {f\left( {\left| {S_j } \right|^2 } \right)x,x} \right\rangle  + \left( {1 - t} \right)\left\langle {f\left( {\left| {T_j  } \right|^2  } \right)x,x} \right\rangle . 
	\end{align*}
	Integrating the above inequality with respect to $t$ over $\left[0,1\right]$, we obtain
	\begin{align}
		&\int_0^1 {f\left( {t\left\langle {\left| {S_j } \right|^2 x,x} \right\rangle  + \left( {1 - t} \right)\left\langle {\left| {T_j  } \right|^2  x,x} \right\rangle } \right)dt}  
		\nonumber\\ 
		&\le \left\langle {\int_0^1 {\left\{ {f\left( {t\left| {S_j } \right|^2 + \left( {1 - t} \right)\left| {T_j  } \right|^2 } \right)dt} \right\}x} ,x} \right\rangle  
		\nonumber\\
		&\le \left\langle {\left( {\frac{{f\left( {\left| {S_j } \right|^2  } \right) + f\left( {\left| {T_j  } \right|^2 } \right)}}{2}} \right)x,x} \right\rangle.  \label{eq2.9}
	\end{align}
	Combining \eqref{eq2.8} and \eqref{eq2.9}, we get
	\begin{align*}
		f\left( {\left| {\left\langle {T_j^*S_j x,x} \right\rangle } \right|} \right) &\le \left\langle {\int_0^1 {\left\{ {f\left( {t\left| {S_j } \right|^2 + \left( {1 - t} \right)\left| {T_j  } \right|^2 } \right)dt} \right\}x} ,x} \right\rangle  
		\\
		&\le \left\langle {\left( {\frac{{f\left( {\left| {S_j } \right|^2  } \right) + f\left( {\left| {T_j  } \right|^2  } \right)}}{2}} \right)x,x} \right\rangle 
	\end{align*}
	and this implies that
	\begin{align*}
		\sum\limits_{j = 1}^n {f\left( {\left| {\left\langle {T^*_jS_j x,x} \right\rangle } \right|} \right)}  
		&\le \sum\limits_{j = 1}^n {\left\langle {\int_0^1 {\left\{ {f\left( {t\left| {S_j } \right|^2   + \left( {1 - t} \right)\left| {T_j  } \right|^2  } \right)dt} \right\}x} ,x} \right\rangle }  
		\\ 
		&\le \sum\limits_{j = 1}^n {\left\langle {\left( {\frac{{f\left( {\left| {S_j } \right|^2 } \right) + f\left( {\left| {T_j  } \right|^2  } \right)}}{2}} \right)x,x} \right\rangle. }   
	\end{align*}
	Since $f^{-1}$ is increasing, we have
	\begin{align*}
		f^{ - 1} \left( {\sum\limits_{j = 1}^n {f\left( {\left| {\left\langle {T^*_jS_j x,x} \right\rangle } \right|} \right)} } \right) 
		&\le f^{ - 1} \left( {\sum\limits_{j = 1}^n {\left\langle {\int_0^1 {\left\{ {f\left( {t\left| {S_j } \right|^2 + \left( {1 - t} \right)\left| {T_j  } \right|^2  } \right)dt} \right\}} x,x} \right\rangle } } \right) 
		\\ 
		&\le f^{ - 1} \left( {\frac{1}{2}\sum\limits_{j = 1}^n {\left\langle {\left[ {f\left( {\left| {S_j } \right|^2  } \right) + f\left( {\left| {T_j  } \right|^2  } \right)} \right]x,x} \right\rangle } } \right)  
	\end{align*}
which is analogous to write
	\begin{align*}
		f^{ - 1} \left( {\sum\limits_{j = 1}^n {f\left( {\left| {\left\langle {S_j x,x} \right\rangle } \right|} \right)} } \right) &\le f^{ - 1} \left( {\left\langle {\left[ {\int_0^1 {\left( {\sum\limits_{j = 1}^n {f\left( {t\left| {S_j } \right|^2   + \left( {1 - t} \right)\left| {T_j  } \right|^2  } \right)} } \right)dt} } \right]x,x} \right\rangle } \right) 
		\\ 
		&\le f^{ - 1} \left( {\left\langle {\left( {\sum\limits_{j = 1}^n {\frac{{f\left( {\left| {S_j } \right|^2  } \right) + f\left( {\left| {T_j  } \right|^2 } \right)}}{2}} } \right)x,x} \right\rangle } \right). 
	\end{align*}
	The desired result follows by taking the supremum over all unit vector $x\in \mathscr{H}$. 
\end{proof}	
\begin{remark}
When $n=1,$ Theorem \ref{thm_squared} gives a refinement and a generalization of \eqref{eq_kitt3}.
\end{remark}



\subsection*{Declarations}
\begin{itemize}
\item {\bf{Availability of data and materials}}: Not applicable.
\item {\bf{Competing interests}}: The authors declare that they have no competing interests.
\item {\bf{Funding}}: Not applicable.
\item {\bf{Authors' contributions}}: Authors declare that they have contributed equally to this paper. All authors have read and approved this version.
\end{itemize}

\vskip 0.5 true cm

\noindent{\tiny (M. Sababheh) Department of Basic Sciences, Princess Sumaya University for Technology, Amman, Jordan}
	
\noindent	{\tiny\textit{E-mail address:} sababheh@yahoo.com; sababheh@psut.edu.jo}

\vskip 0.3 true cm

\noindent{\tiny (H. R. Moradi) Department of Mathematics, Mashhad Branch, Islamic Azad University, Mashhad, Iran
	
\noindent	\textit{E-mail address:} hrmoradi@mshdiau.ac.ir}

\vskip 0.3 true cm

\noindent{\tiny (M. W. Alomari) Department of Mathematics, Faculty of Science and Information Technology, Irbid National University,
Irbid 21110, Jordan}
	
\noindent	{\tiny\textit{E-mail address:} mwomath@gmail.com}

\end{document}